\documentclass[pdflatex,sn-mathphys-num,oneside]{sn-jnl}
\makeatletter
\@twosidefalse
\@mparswitchfalse
\makeatother
\usepackage{amssymb,mathtools,pifont}
\usepackage{algorithm}
\usepackage{algpseudocode}
\usepackage{booktabs,array,tabularx}
\usepackage{enumitem,microtype,needspace,placeins}
\usepackage[nameinlink,noabbrev]{cleveref}

\theoremstyle{thmstyleone}
\newtheorem{theorem}{Theorem}[section]
\newtheorem{lemma}[theorem]{Lemma}

\theoremstyle{thmstyletwo}
\newtheorem{remark}[theorem]{Remark}
\theoremstyle{thmstylethree}
\newtheorem{definition}[theorem]{Definition}
\newtheorem{assumption}[theorem]{Assumption}
\newcommand{\R}{\mathbb R}
\newcommand{\eps}{\varepsilon}
\newcommand{\ip}[2]{\langle #1,#2\rangle}
\newcommand{\norm}[1]{\lVert #1\rVert}

\allowdisplaybreaks[4]
\title[Near-Optimal Single-Loop Extragradient Method]
{Near-Optimal Single-Loop Predictor--Corrector Extragradient Method
	for Strongly Convex--Strongly Concave Minimax Optimization}
\author[1]{\fnm{Minhao} \sur{Zhang}}
\email{zhangminhao@shu.edu.cn}
\author*[1]{\fnm{Zi} \sur{Xu}}
\email{xuzi@shu.edu.cn}
\affil*[1]{
	\orgdiv{Department of Mathematics},
	\orgname{Shanghai University},
	\orgaddress{
		\city{Shanghai},
		\postcode{200444},
		\country{People's Republic of China}
	}
}
\begin{document}
	\abstract{
We study smooth strongly convex--strongly concave minimax optimization
in the deterministic unconstrained setting, without assuming a bilinear
or separable structure.
Although existing multi-loop methods attain near-optimal
condition-number dependence, standard single-loop methods generally
exhibit a substantial complexity gap.
To close this gap, we propose the Single-Loop Predictor--Corrector
Extragradient Method with Damped Momentum (PCE-DM), which combines an
extragradient prediction--correction scheme with a novel auxiliary
feedback recursion for the weaker-curvature variable.
PCE-DM uses fixed parameters and two new full-gradient evaluations per
iteration after one initialization query, while requiring no inner
solves, accuracy schedules, or staged restarts.
We develop a Lyapunov analysis that controls the predictor--corrector
mismatch through corrected-gradient increments and establish
last-iterate linear convergence.
Specifically, PCE-DM computes an $\varepsilon$-accurate relative
solution, measured by the squared Euclidean distance to the saddle
point, within
$\mathcal{O}\!\left(\sqrt{\kappa_x\kappa_y}
\log(2\kappa_x\kappa_y/\varepsilon)\right)$ full-gradient queries.
This result closes the condition-number complexity gap between standard
single-loop methods and near-optimal multi-loop methods, matching the
known lower-bound order up to logarithmic factors while retaining fixed,
explicit single-loop updates.
Numerical experiments on regularized linear regression and AUC
maximization demonstrate the computational efficiency of PCE-DM.
	}
	
	\maketitle
	
	\section{Introduction}\label{sec:introduction}	
	We consider the smooth strongly convex--strongly concave minimax problem
	\begin{equation}\label{eq:intro-problem}
		\min_{x\in\mathbb R^n}\max_{y\in\mathbb R^p} f(x,y),
	\end{equation}
	where $f$ has an $L$-Lipschitz continuous full gradient, is
	$\mu_x$-strongly convex in $x$, and is $\mu_y$-strongly concave in $y$.
	We focus on the deterministic unconstrained setting and do not impose a
	bilinear or separable structure on $f$.
	Strongly convex--strongly concave minimax formulations arise in
	regularized empirical risk minimization and regularized policy
	evaluation in reinforcement learning~\cite{zhangxiao2017,du2017}.
	Related regularized saddle-point models appear in batch policy learning,
	stochastic composite optimization, and mixed-strategy Nash equilibrium
	estimation~\cite{zhanggeneralization2021}.
	
	Let $\kappa_x=L/\mu_x$ and $\kappa_y=L/\mu_y$.
	For general smooth objectives, multi-loop methods based on approximately
	solving auxiliary subproblems attain
	$\widetilde{\mathcal{O}}(\sqrt{\kappa_x\kappa_y})$ first-order
	complexity~\cite{lin2020,wangli2020,kovalev2022}, matching the lower-bound
	order $\widetilde{\Omega}(\sqrt{\kappa_x\kappa_y})$ for the corresponding
	first-order algorithm classes~\cite{zhang2021,ibrahim2020}.
	In contrast, standard single-loop methods have complexity
	$\widetilde{\mathcal{O}}(\kappa_x\kappa_y)$ for GDA and
	$\widetilde{\mathcal{O}}(\max\{\kappa_x,\kappa_y\})$ for EG and
	OGDA~\cite{lee2024,mokhtari2020}, leaving a gap from the lower-bound
	condition-number dependence.
	This work closes this gap, up to logarithmic factors, by developing the
	Single-Loop Predictor--Corrector Extragradient Method with Damped
	Momentum (PCE-DM).
	For the last-iterate squared-distance criterion, PCE-DM attains
	$\widetilde{\mathcal{O}}(\sqrt{\kappa_x\kappa_y})$ complexity using a
	fixed sequence of explicit full-gradient updates.
	Its construction combines an auxiliary feedback recursion with a
	stability argument that controls the residual induced by explicit
	discretization.
	
	\subsection{Related work}\label{subsec:related-work}
	
	We review first-order methods for strongly convex--strongly concave
	minimax problems, focusing on general smooth objectives without a
	bilinear structural assumption.
	We then briefly discuss methods specialized to bilinearly coupled
	separable objectives and the relevant lower bounds.
	Table~\ref{tab:related-work} summarizes the complexity orders and update
	structures of the methods most closely related to this work.
	Here, $L$ is both the full-gradient smoothness constant and an upper
	bound on the block smoothness constants.
	A single-loop method performs a fixed sequence of explicit updates
	without approximately solving inner subproblems.
	The notation $\widetilde{\mathcal{O}}$ and $\widetilde{\Omega}$
	suppresses logarithmic factors in the condition numbers and accuracy.
	
	\paragraph{Single-loop methods for general smooth objectives}
	Representative single-loop methods include gradient descent--ascent
	(GDA)~\cite{zamGDA2024}, extragradient
	(EG)~\cite{korpelevich1976}, and optimistic gradient descent--ascent
	(OGDA)~\cite{popov1980}.
	Under the common smoothness bound, their squared-distance iteration
	complexities are
	$\widetilde{\mathcal{O}}(\kappa_x\kappa_y)$ for GDA with
	curvature-matched block stepsizes~\cite[Corollary~3.2]{lee2024} and
	$\widetilde{\mathcal{O}}(\max\{\kappa_x,\kappa_y\})$ for EG and
	OGDA~\cite[Theorems~7 and~4]{mokhtari2020}.
	These methods have explicit updates and low per-iteration costs, but
	their condition-number dependence does not match the known lower bound,
	particularly when $\kappa_x$ and $\kappa_y$ are highly imbalanced.
	
	\paragraph{Multi-loop methods for general smooth objectives}
	Improved condition-number dependence has been obtained through
	multi-loop or inexact-proximal constructions that approximately solve
	auxiliary subproblems.
	Minimax-APPA~\cite[Theorem~9]{lin2020} combines accelerated
	proximal-point iterations with an accelerated solver for each minimax
	subproblem. Its guarantee is established over bounded convex domains,
	whose diameters enter the logarithmic factors.
	Proximal Best Response~\cite[Theorem~3]{wangli2020} distinguishes
	interaction smoothness from within-variable smoothness and, under the
	common bound $L$, attains
	$\widetilde{\mathcal{O}}(\sqrt{\kappa_x\kappa_y})$ complexity through
	nested updates.
	FOAM~\cite{kovalev2022} employs a pointwise-conjugate reformulation and
	inexact proximal updates to attain the optimal first-order complexity.
	For composite non-bilinear saddle-point problems, the accelerated
	sliding method in~\cite{alkousa2020accelerated} exploits separately
	available component gradients and inexact oracles.
	Thus, the optimal condition-number order is attainable for general
	smooth objectives, but existing constructions achieving this order
	require nested iterations, inexact subproblem solutions, or additional
	oracle structures.
	
	\paragraph{Lower bounds}
	The lower bounds in~\cite{zhang2021,ibrahim2020} establish the benchmark
	$\widetilde{\Omega}(\sqrt{\kappa_x\kappa_y})$ after constant-factor
	normalization of the block smoothness parameters.
	They apply to the respective classes of first-order algorithms, with
	the dimension allowed to grow, and do not distinguish between
	single-loop and nested implementations.
	Consequently, these results determine the target condition-number
	dependence but do not imply that it can be attained by explicit
	single-loop updates.
	PCE-DM realizes this dependence, up to logarithmic factors, without
	inner solves or restarts.
	
	\paragraph{Methods for bilinearly coupled objectives}
	For the more structured objective
	$a(x)+\langle Ax,y\rangle-b(y)$, where $a$ and $b$ are smooth and
	strongly convex, the Lifted Primal-Dual (LPD) method~\cite{lpd2022} and
	the Accelerated Primal-Dual Gradient (APDG) method~\cite{apdg2022}
	attain optimal rates with single-loop updates.
	Their bounds distinguish component smoothness from coupling constants
	and reduce to
	$\widetilde{\mathcal{O}}(\sqrt{\kappa_x\kappa_y})$ under the common
	bound $L$.
	These methods exploit the bilinear coupling through component-gradient
	and matrix-vector oracles.
	Primal--dual extragradient schemes also provide sharp guarantees for
	broader separable models~\cite{jin2022}.
	In contrast, PCE-DM uses only full-gradient queries and does not require
	the objective to have a bilinearly coupled separable form.
	\begin{table}
		\centering
		\caption{Complexity orders and update structures for general smooth
			strongly convex--strongly concave minimax optimization.}
		\label{tab:related-work}
		\small
		\newcommand{\tablecheck}{{\large\ding{51}}}
		\newcommand{\tablecross}{{\large\ding{55}}}
		\setlength{\tabcolsep}{3pt}
		\renewcommand{\arraystretch}{1.35}
		\begin{tabular}{@{}p{0.47\textwidth}c c@{}}
			\toprule
			Method
			& Complexity
			& Single-loop\\
			\midrule
			GDA~\cite[Cor.~3.2]{lee2024}
			& $\widetilde{\mathcal{O}}(\kappa_x\kappa_y)$
			& \tablecheck\\
			EG~\cite[Thm.~7]{mokhtari2020}
			& $\widetilde{\mathcal{O}}(\max\{\kappa_x,\kappa_y\})$
			& \tablecheck\\
			OGDA~\cite[Thm.~4]{mokhtari2020}
			& $\widetilde{\mathcal{O}}(\max\{\kappa_x,\kappa_y\})$
			& \tablecheck\\
			\midrule
			Minimax-APPA~\cite[Thm.~9]{lin2020}
			& $\widetilde{\mathcal{O}}(\sqrt{\kappa_x\kappa_y})$
			& \tablecross\\
			Proximal Best Response\newline
			\cite[Thm.~3]{wangli2020}
			& $\widetilde{\mathcal{O}}(\sqrt{\kappa_x\kappa_y})$
			& \tablecross\\
			FOAM~\cite[Thm.~3, Cor.~1]{kovalev2022}
			& $\widetilde{\mathcal{O}}(\sqrt{\kappa_x\kappa_y})$
			& \tablecross\\
			\midrule
			\textbf{PCE-DM}
			& $\widetilde{\mathcal{O}}(\sqrt{\kappa_x\kappa_y})$
			& \tablecheck\\
			\midrule
			Lower bounds~\cite{zhang2021,ibrahim2020}
			& $\widetilde{\Omega}(\sqrt{\kappa_x\kappa_y})$
			& ---\\
			\bottomrule
		\end{tabular}
	\end{table}
	
	\subsection{Our contributions}\label{subsec:contributions}
	Our main contributions are as follows.
	\begin{enumerate}[label=(\arabic*),leftmargin=*,itemsep=0.6em]
		\item \emph{A single-loop algorithm with near-optimal complexity.}
		We propose PCE-DM for smooth strongly convex--strongly concave minimax
		problems without requiring a bilinearly coupled separable structure.
		The method uses fixed parameters and two new full-gradient evaluations
		per iteration, requiring neither inner solves nor restarts.
		With appropriate parameter choices, its last iterate is an
		$\varepsilon$-accurate relative solution in the sense of
		Definition~\ref{def:accuracy} after
		\begin{equation}\label{eq:intro-complexity}
			\mathcal{O}\!\left(
			\sqrt{\kappa_x\kappa_y}
			\log\frac{2\kappa_x\kappa_y}{\varepsilon}
			\right)
		\end{equation}
		full-gradient queries.
		Thus, PCE-DM matches the known lower-bound condition-number dependence,
		up to logarithmic factors, while retaining an explicit single-loop
		update structure.
		
		\item \emph{A new Lyapunov analysis.}
		We construct a Lyapunov function that combines the function-value gap,
		squared gradient-direction norms, momentum terms, and distances to the
		saddle point.
		Its lower bound controls the squared distance to the saddle point,
		while its one-step descent estimate produces negative
		gradient-increment terms that absorb the predictor--corrector mismatch.
		Combined with the damped-momentum recursion, this estimate yields a
		contraction and establishes last-iterate linear convergence.
		
		\item \emph{Numerical evidence of computational efficiency.}
		We compare PCE-DM with single-loop and nested baselines on regularized
		linear regression and AUC maximization problems.
		PCE-DM substantially reduces the number of full-gradient queries on
		the tested curvature-imbalanced regression instances and on the a9a
		and cod-rna datasets.
		It also achieves the shortest measured computation time on these
		regression instances and both AUC datasets.
	\end{enumerate}	
\paragraph{Organization}
Section~\ref{sec:algorithm} introduces PCE-DM and its update structure.
Section~\ref{sec:analysis} establishes last-iterate linear convergence
and derives the resulting complexity bound.
Section~\ref{sec:experiments} presents the numerical results, and
Section~\ref{sec:conclusion} concludes the paper.
\section{PCE-DM Algorithm}\label{sec:algorithm}
We now introduce the Single-Loop Predictor--Corrector Extragradient
Method with Damped Momentum (PCE-DM). As discussed in
Section~\ref{sec:introduction}, standard single-loop methods, including
gradient descent--ascent, extragradient, and optimistic gradient
descent--ascent, typically employ symmetric update structures for the
primal and dual variables
\cite{zamGDA2024,korpelevich1976,popov1980}. Although these methods have
explicit and computationally inexpensive iterations, their symmetric
structures do not directly exploit an imbalance between the primal and
dual curvature parameters. Our goal is to retain the simplicity of a
fixed-parameter, explicit single-loop method while achieving the
near-optimal condition-number dependence discussed in the introduction.

The main idea of PCE-DM is to combine the extragradient
prediction--correction framework with an auxiliary feedback recursion
adapted to the weaker-curvature variable. Under the ordering
$\mu_x\ge\mu_y$, the dual variable $y$ has weaker curvature, and we
therefore introduce the auxiliary momentum variable $v_t$ into its
update. The predictor step uses the current feedback $v_t$ to estimate
the dual displacement, while the corrector step evaluates the gradient
at the predicted point and incorporates the provisional feedback
$\widetilde v_t$. The auxiliary variable is then refreshed at the
corrected point, so the provisional and final feedback updates share the
same damped history term $\alpha v_t$. This construction yields fixed,
explicit updates and requires neither inner subproblem solves nor
accuracy schedules or staged restarts. The complete method is given in
Algorithm \ref{alg:main}.
\begin{algorithm}[h]
	\caption{Single-Loop Predictor--Corrector Extragradient
		Method with Damped Momentum (PCE-DM)}\label{alg:main}
	\begin{algorithmic}[1]
		\Require Initial point $(x_0,y_0)\in\R^n\times\R^p$,
		iteration count $T$, and parameters $h,\alpha,\rho$.
		\State Set $v_0=0$.
		\For{$t=0,1,\ldots,T-1$}
		\State $\displaystyle
		\widetilde x_t=x_t-h\nabla_x f(x_t,y_t)$
		\State $\displaystyle
		\widetilde y_t
		=y_t+h\bigl(\nabla_y f(x_t,y_t)+v_t\bigr)$
		\State $\displaystyle
		\widetilde v_t
		=\alpha v_t
		+\rho\nabla_y f(\widetilde x_t,\widetilde y_t)$
		\State $\displaystyle
		x_{t+1}
		=x_t-h\nabla_x f(\widetilde x_t,\widetilde y_t)$
		\State $\displaystyle
		y_{t+1}
		=y_t+h\bigl(
		\nabla_y f(\widetilde x_t,\widetilde y_t)
		+\widetilde v_t
		\bigr)$
		\State $\displaystyle
		v_{t+1}
		=\alpha v_t
		+\rho\nabla_y f(x_{t+1},y_{t+1})$
		\EndFor
		\Ensure $(x_T,y_T)$.
	\end{algorithmic}
\end{algorithm}
The role of the auxiliary feedback recursion can be seen more explicitly
by simplifying the dual correction update:
\begin{align*}
	y_{t+1}
	&=y_t+h\Bigl(
	\nabla_y f(\widetilde x_t,\widetilde y_t)+\widetilde v_t
	\Bigr)\\
	&=y_t+h\Bigl(
	(1+\rho)
	\nabla_y f(\widetilde x_t,\widetilde y_t)
	+\alpha v_t
	\Bigr)\\
	&=y_t+\alpha(\widetilde y_t-y_t)
	+h\Bigl(
	(1+\rho)
	\nabla_y f(\widetilde x_t,\widetilde y_t)
	-\alpha\nabla_y f(x_t,y_t)
	\Bigr).
\end{align*}
The second equality substitutes the definition of $\widetilde v_t$,
whereas the last equality uses the predictor relation
$\widetilde y_t-y_t=h\bigl(\nabla_y f(x_t,y_t)+v_t\bigr)$.
Consequently, $\widetilde y_t$ predicts the dual displacement using the
current gradient and accumulated feedback, while $y_{t+1}$ corrects this
prediction using the gradient evaluated at
$(\widetilde x_t,\widetilde y_t)$. The shared history term in
$\widetilde v_t$ and $v_{t+1}$ is also central to the residual-control
argument used in the Lyapunov analysis of the next section.

\section{Lyapunov analysis and complexity result}\label{sec:analysis}
We first state the assumptions under which the convergence analysis is
carried out.
\begin{assumption}\label[assumption]{ass:main}
	The function $f:\R^n\times\R^p\to\R$ is continuously differentiable
	and satisfies the following conditions.
	\begin{enumerate}[label=(\roman*),leftmargin=2em]
		\item The full gradient $\nabla f$ is $L$-Lipschitz continuous for
		some $L>0$; that is, for all
		$(x,y),(x',y')\in\R^n\times\R^p$,
		\begin{equation}\label{eq:smooth}
			\norm{\nabla f(x',y')-\nabla f(x,y)}
			\le L\norm{(x'-x,y'-y)}.
		\end{equation}
		
		\item There exist constants $\mu_x,\mu_y>0$ such that, for all
		$x,x'\in\R^n$ and $y,y'\in\R^p$,
		\begin{align}
			f(x',y)
			&\ge f(x,y)+\ip{\nabla_x f(x,y)}{x'-x}
			+\frac{\mu_x}{2}\norm{x'-x}^2,
			\label{eq:sc}\\
			f(x,y')
			&\le f(x,y)+\ip{\nabla_y f(x,y)}{y'-y}
			-\frac{\mu_y}{2}\norm{y'-y}^2.
			\label{eq:scave}
		\end{align}
	\end{enumerate}
\end{assumption}
Throughout the analysis, we assume without loss of generality that
\[
L\ge\mu_x\ge\mu_y>0.
\]
The ordering $\mu_x\ge\mu_y$ identifies $y$ as the variable with weaker
curvature and motivates the asymmetric feedback mechanism in PCE-DM.
If $\mu_y>\mu_x$, the roles of the primal and dual variables may be
interchanged by applying the method to
$\widehat f(\xi,\eta)=-f(\eta,\xi)$, $\xi\in\R^p$, $\eta\in\R^n.$
Under Assumption \ref{ass:main}, problem~\eqref{eq:intro-problem} admits a unique
saddle point $(x^\star,y^\star)$ satisfying $\nabla f(x^\star,y^\star)=0.$
We denote the corresponding saddle value by
$f^\star=f(x^\star,y^\star)$; see
\cite[Section~2]{wangli2020}.
We measure convergence in terms of the squared Euclidean distance to
the unique saddle point.
\begin{definition}[$\eps$-accurate relative solution]
	\label[definition]{def:accuracy}
	For an iterate $(x_t,y_t)$, define
	\[
	D_t
	:=\norm{x_t-x^\star}^2+\norm{y_t-y^\star}^2.
	\]
	Given $0<\eps<1$, an iterate $(x_T,y_T)$ is called an
	$\eps$-accurate relative solution if $D_T\le\eps D_0.$
\end{definition}
We now establish the last-iterate linear convergence of PCE-DM under
Assumotion \ref{ass:main} and derive the resulting full-gradient complexity bound.
Throughout this section, the algorithmic parameters satisfy
\begin{equation}\label{eq:analysis-parameters}
	0<h\le \frac{2}{15L},\qquad
	\rho=\frac{h\mu_x}{8},\qquad
	\frac{4}{4+h\sqrt{\mu_x\mu_y}}
	\le \alpha
	\le \frac{8}{8+h\sqrt{\mu_x\mu_y}}.
\end{equation}
For notational convenience, define the predictor and corrector
directions by
\begin{alignat}{2}
	d_t^x&:=\nabla_x f(x_t,y_t),
	&\qquad d_t^y&:=\nabla_y f(x_t,y_t)+v_t,
	\label{eq:predictor-directions}\\*
	\widetilde d_t^x&:=\nabla_x f(\widetilde x_t,\widetilde y_t),
	&\qquad
	\widetilde d_t^y&:=\nabla_y f(\widetilde x_t,\widetilde y_t)
	+\widetilde v_t.
	\label{eq:corrector-directions}
\end{alignat}
The feedback recursion in Algorithm \ref{alg:main} can equivalently be expressed
as
\begin{equation}\label{eq:endv}
	v_{t+1}-v_t
	=\frac{\rho}{\alpha}\nabla_y f(x_{t+1},y_{t+1})
	-\frac{1-\alpha}{\alpha}v_{t+1}.
\end{equation}
Our analysis is based on the following Lyapunov function:
\begin{align}
	\mathcal E_t
	&:=-\frac{1-\alpha+2\rho}{2\alpha h}
	\bigl(f(x_t,y_t)-f^\star\bigr)
	+\frac12\norm{d_t^x}^2
	+\frac12\norm{d_t^y}^2
	\notag\\
	&\qquad
	+\frac{1-\alpha}{2\alpha h}\ip{v_t}{y_t-y^\star}
	+\frac{(1-\alpha)^2}{4\alpha^2h^2}
	\norm{y_t-y^\star}^2.
	\label{eq:lyapunov}
\end{align}
We first establish that $\mathcal E_t$ is coercive with respect to the
distance to the saddle point. We then prove a one-step descent inequality
that controls the predictor--corrector mismatch through successive
direction increments. Combining these two properties yields the desired
linear convergence and complexity result.

\begin{lemma}\label{lem:coercivity}
	Under Assumption~\ref{ass:main} and the parameter conditions
	\eqref{eq:analysis-parameters}, for every $t\ge0$,
	\begin{equation}\label{eq:coercivity-simple}
		\mathcal E_t\ge \frac{\sqrt{\mu_x\mu_y}}{32}
		\bigl(\mu_x\norm{x_t-x^\star}^2
		+\mu_y\norm{y_t-y^\star}^2\bigr)\ge0.
	\end{equation}
\end{lemma}

\begin{proof}
	The parameter conditions \eqref{eq:analysis-parameters} imply
	$0<\alpha<1$ and $\rho>0$, so $1-\alpha+2\rho\ge0$.
	Strong concavity gives
	\begin{equation*}
		f(x_t,y^\star)\le f(x_t,y_t)
		-\ip{\nabla_y f(x_t,y_t)}{y_t-y^\star}
		-\frac{\mu_y}{2}\norm{y_t-y^\star}^2,
	\end{equation*}
	while strong convexity and $\nabla_x f(x^\star,y^\star)=0$ give
	\begin{equation*}
		f(x_t,y^\star)\ge f^\star
		+\frac{\mu_x}{2}\norm{x_t-x^\star}^2.
	\end{equation*}
	Combining these inequalities yields
	\begin{equation}\label{eq:geom3}
		f(x_t,y_t)-f^\star
		-\ip{\nabla_y f(x_t,y_t)}{y_t-y^\star}
		\ge \frac{\mu_x}{2}\norm{x_t-x^\star}^2
		+\frac{\mu_y}{2}\norm{y_t-y^\star}^2.
	\end{equation}
	Next, strong convexity and Young's inequality give
	\begin{equation*}
		f(x_t,y_t)-f(x^\star,y_t)
		\le \ip{d_t^x}{x_t-x^\star}
		-\frac{\mu_x}{2}\norm{x_t-x^\star}^2
		\le \frac{1}{2\mu_x}\norm{d_t^x}^2.
	\end{equation*}
	Strong concavity and $\nabla_y f(x^\star,y^\star)=0$ also give
	\begin{equation*}
		f(x^\star,y_t)-f^\star
		\le -\frac{\mu_y}{2}\norm{y_t-y^\star}^2.
	\end{equation*}
	Adding these inequalities yields
	\begin{equation}\label{eq:geom4}
		f(x_t,y_t)-f^\star
		\le \frac{1}{2\mu_x}\norm{d_t^x}^2
		-\frac{\mu_y}{2}\norm{y_t-y^\star}^2.
	\end{equation}
	Substituting $v_t=d_t^y-\nabla_y f(x_t,y_t)$ and completing the square,
	we obtain
	\begin{equation*}
		\begin{aligned}
			&\frac12\norm{d_t^y}^2
			+\frac{1-\alpha}{2\alpha h}\ip{v_t}{y_t-y^\star}
			+\frac{(1-\alpha)^2}{4\alpha^2h^2}\norm{y_t-y^\star}^2\\
			={}&\frac12\norm{d_t^y
				+\frac{1-\alpha}{2\alpha h}(y_t-y^\star)}^2
			+\frac{(1-\alpha)^2}{8\alpha^2h^2}\norm{y_t-y^\star}^2
			-\frac{1-\alpha}{2\alpha h}
			\ip{\nabla_y f(x_t,y_t)}{y_t-y^\star}.
		\end{aligned}
	\end{equation*}
	Moreover, \eqref{eq:geom3} and \eqref{eq:geom4} imply
	\begin{equation}\label{eq:dual-inner-upper}
		\ip{\nabla_y f(x_t,y_t)}{y_t-y^\star}
		\le \frac{1}{2\mu_x}\norm{d_t^x}^2
		-\frac{\mu_x}{2}\norm{x_t-x^\star}^2
		-\mu_y\norm{y_t-y^\star}^2.
	\end{equation}
	Using the square-completion identity in \eqref{eq:lyapunov}, applying
	\eqref{eq:geom4} and \eqref{eq:dual-inner-upper}, and dropping
	nonnegative terms yields
	\begin{align*}
		\mathcal E_t\ge{}&\left[\frac12
		-\frac{1-\alpha+2\rho}{4\alpha h\mu_x}
		-\frac{1-\alpha}{4\alpha h\mu_x}\right]\norm{d_t^x}^2\\
		&+\frac{1-\alpha}{4\alpha h}
		\left(\mu_x\norm{x_t-x^\star}^2
		+\mu_y\norm{y_t-y^\star}^2\right).
	\end{align*}
	Under \eqref{eq:analysis-parameters}, the coefficient of
	$\norm{d_t^x}^2$ is nonnegative. Thus,
	\begin{equation*}
		\mathcal E_t\ge \frac{1-\alpha}{4\alpha h}
		\left(\mu_x\norm{x_t-x^\star}^2
		+\mu_y\norm{y_t-y^\star}^2\right).
	\end{equation*}
	Finally, the upper bound on $\alpha$ in
	\eqref{eq:analysis-parameters} gives \eqref{eq:coercivity-simple}.
\end{proof}

To establish one-step Lyapunov descent, we introduce the following
notation and two auxiliary lemmas. Define
\begin{equation*}
	\gamma:=\frac{1-\alpha}{2\alpha h},
	\qquad
	\beta:=\frac{1-\alpha+2\rho}{2\alpha h}.
\end{equation*}
The parameter conditions \eqref{eq:analysis-parameters} imply
\begin{equation}
	\frac{\sqrt{\mu_x\mu_y}}{16}
	\le\gamma\le\frac{\sqrt{\mu_x\mu_y}}8,
	\qquad
	\frac{\mu_x}{8\alpha}+\frac{\sqrt{\mu_x\mu_y}}{16}
	\le\beta\le\frac{\mu_x}{8\alpha}+\frac{\sqrt{\mu_x\mu_y}}8.
	\label{eq:descent-parameters}
\end{equation}
\begin{lemma}\label{lem:correction-geometry}
	Under Assumption \ref{ass:main}, the correction updates in Algorithm \ref{alg:main}
	satisfy, for every $t\ge0$,
	\begin{equation}\label{eq:correction-geometry}
		\ip{\widetilde d_t^x}{d_{t+1}^x-d_t^x}
		+\ip{\widetilde d_t^y}
		{\nabla_y f(x_{t+1},y_{t+1})-\nabla_y f(x_t,y_t)}
		\le-h\mu_x\norm{\widetilde d_t^x}^2
		-h\mu_y\norm{\widetilde d_t^y}^2.
	\end{equation}
\end{lemma}
\begin{proof}
	Strong convexity--strong concavity implies the following strong
	monotonicity inequality for all
	$(x,y),(x',y')\in\R^n\times\R^p$:
	\begin{align*}
		&\ip{\nabla_x f(x',y')-\nabla_x f(x,y)}{x'-x}
		-\ip{\nabla_y f(x',y')-\nabla_y f(x,y)}{y'-y}\\
		\ge{}&\mu_x\norm{x'-x}^2+\mu_y\norm{y'-y}^2.
	\end{align*}
	Applying this inequality with $(x',y')=(x_{t+1},y_{t+1})$ and
	$(x,y)=(x_t,y_t)$, and using the correction relations
	\begin{equation*}
		x_{t+1}-x_t=-h\widetilde d_t^x,
		\qquad
		y_{t+1}-y_t=h\widetilde d_t^y,
	\end{equation*}
	yields \eqref{eq:correction-geometry}, since $h>0$.
\end{proof}
\begin{lemma}[Predictor--corrector mismatch]\label{lem:defect}
	Under Assumption \ref{ass:main} and the parameter conditions
	\eqref{eq:analysis-parameters}, for every $t\ge0$,
	\begin{equation}\label{eq:defect-bound}
		\norm{\widetilde d_t^x-d_{t+1}^x}^2
		+\norm{\widetilde d_t^y-d_{t+1}^y}^2
		\le\frac1{36}\left(
		\norm{d_{t+1}^x-d_t^x}^2
		+\norm{d_{t+1}^y-d_t^y}^2\right).
	\end{equation}
\end{lemma}
\begin{proof}
	The definitions of the dual directions and the cancellation of the
	shared history term $\alpha v_t$ in the feedback updates give
	\begin{equation*}
		\widetilde d_t^y-d_{t+1}^y
		=(1+\rho)\bigl(
		\nabla_y f(\widetilde x_t,\widetilde y_t)
		-\nabla_y f(x_{t+1},y_{t+1})\bigr).
	\end{equation*}
	Since $1+\rho\ge1$, Lipschitz continuity of $\nabla f$ implies
	\begin{align*}
		\sqrt{\norm{\widetilde d_t^x-d_{t+1}^x}^2
			+\norm{\widetilde d_t^y-d_{t+1}^y}^2}
		&\le(1+\rho)
		\norm{\nabla f(\widetilde x_t,\widetilde y_t)
			-\nabla f(x_{t+1},y_{t+1})}\\
		&\le L(1+\rho)
		\sqrt{\norm{\widetilde x_t-x_{t+1}}^2
			+\norm{\widetilde y_t-y_{t+1}}^2}.
	\end{align*}
	The predictor and corrector updates imply
	\begin{equation*}
		\widetilde x_t-x_{t+1}
		=h\bigl(d_{t+1}^x-d_t^x+\widetilde d_t^x-d_{t+1}^x\bigr),
		\qquad
		\widetilde y_t-y_{t+1}
		=-h\bigl(d_{t+1}^y-d_t^y+\widetilde d_t^y-d_{t+1}^y\bigr).
	\end{equation*}
	Applying the triangle inequality therefore yields
	\begin{align*}
		&\sqrt{\norm{\widetilde x_t-x_{t+1}}^2
			+\norm{\widetilde y_t-y_{t+1}}^2}\\
		\le{}&h\sqrt{\norm{d_{t+1}^x-d_t^x}^2
			+\norm{d_{t+1}^y-d_t^y}^2}
		+h\sqrt{\norm{\widetilde d_t^x-d_{t+1}^x}^2
			+\norm{\widetilde d_t^y-d_{t+1}^y}^2}.
	\end{align*}
	By \eqref{eq:analysis-parameters} and $\mu_x\le L$,
	$\rho=h\mu_x/8\le1/60$ and $hL(1+\rho)\le61/450<1/7$.
	Combining the preceding bounds and rearranging gives
	\begin{align*}
		\sqrt{\norm{\widetilde d_t^x-d_{t+1}^x}^2
			+\norm{\widetilde d_t^y-d_{t+1}^y}^2}
		\le{}&\frac{hL(1+\rho)}{1-hL(1+\rho)}
		\sqrt{\norm{d_{t+1}^x-d_t^x}^2
			+\norm{d_{t+1}^y-d_t^y}^2}\\
		\le{}&\frac16
		\sqrt{\norm{d_{t+1}^x-d_t^x}^2
			+\norm{d_{t+1}^y-d_t^y}^2}.
	\end{align*}
	Squaring both sides proves \eqref{eq:defect-bound}.
\end{proof}
\begin{lemma}[One-step Lyapunov descent]\label{lem:descent}
	Under Assumption \ref{ass:main} and the parameter conditions
	\eqref{eq:analysis-parameters}, for every $t\ge0$,
	\begin{equation}\label{eq:main-descent}
		\frac{1+\alpha}{2\alpha}\mathcal E_{t+1}-\mathcal E_t
		\le -\frac18\left(
		\norm{d_{t+1}^x-d_t^x}^2+\norm{d_{t+1}^y-d_t^y}^2\right).
	\end{equation}
\end{lemma}

\begin{proof}
	With the notation $\gamma$ and $\beta$, \eqref{eq:lyapunov} reads
	\begin{equation*}
		\mathcal E_t
		=-\beta\bigl(f(x_t,y_t)-f^\star\bigr)
		+\frac12\norm{d_t^x}^2+\frac12\norm{d_t^y}^2
		+\gamma\ip{v_t}{y_t-y^\star}
		+\gamma^2\norm{y_t-y^\star}^2.
	\end{equation*}
	Expanding the weighted Lyapunov difference gives
	\begin{align*}
		(1+h\gamma)\mathcal E_{t+1}-\mathcal E_t
		=&-\beta\bigl[f(x_{t+1},y_{t+1})-f(x_t,y_t)\bigr]
		-h\gamma\beta\bigl[f(x_{t+1},y_{t+1})-f^\star\bigr]\\
		&+\frac12\sum_{i\in\{x,y\}}\bigl(\norm{d_{t+1}^i}^2-\norm{d_t^i}^2\bigr)
		+\frac{h\gamma}{2}\sum_{i\in\{x,y\}}\norm{d_{t+1}^i}^2\\
		&+\gamma\bigl[\ip{v_{t+1}}{y_{t+1}-y^\star}
		-\ip{v_t}{y_t-y^\star}\bigr]
		+h\gamma^2\ip{v_{t+1}}{y_{t+1}-y^\star}\\
		&+\gamma^2\bigl[\norm{y_{t+1}-y^\star}^2-\norm{y_t-y^\star}^2\bigr]
		+h\gamma^3\norm{y_{t+1}-y^\star}^2.
	\end{align*}
	For each $i\in\{x,y\}$, the difference of squared norms satisfies
	\begin{equation*}
		\frac12\bigl(\norm{d_{t+1}^i}^2-\norm{d_t^i}^2\bigr)
		=\ip{\widetilde d_t^i}{d_{t+1}^i-d_t^i}
		-\ip{\widetilde d_t^i-d_{t+1}^i}{d_{t+1}^i-d_t^i}
		-\frac12\norm{d_{t+1}^i-d_t^i}^2.
	\end{equation*}
	The dual correction relation $y_{t+1}-y_t=h\widetilde d_t^y$ gives
	\begin{equation*}
		\ip{v_{t+1}}{y_{t+1}-y^\star}
		-\ip{v_t}{y_t-y^\star}
		=\ip{v_{t+1}-v_t}{y_{t+1}-y^\star}
		+h\ip{v_{t+1}}{\widetilde d_t^y}
		-h\ip{v_{t+1}-v_t}{\widetilde d_t^y},
	\end{equation*}
	and
	\begin{equation*}
		\norm{y_{t+1}-y^\star}^2-\norm{y_t-y^\star}^2
		=2h\ip{\widetilde d_t^y}{y_{t+1}-y^\star}
		-h^2\norm{\widetilde d_t^y}^2.
	\end{equation*}
	Moreover, the definition of $d_t^y$ and \eqref{eq:endv} imply
	\begin{align*}
		d_{t+1}^y-d_t^y
		&=\nabla_y f(x_{t+1},y_{t+1})-\nabla_y f(x_t,y_t)+v_{t+1}-v_t,\\
		v_{t+1}-v_t+h\gamma v_{t+1}
		&=h\beta\nabla_y f(x_{t+1},y_{t+1})-h\gamma d_{t+1}^y.
	\end{align*}
	Substituting these identities into the expansion and using
	\begin{equation*}
		\frac12\norm{d_{t+1}^y}^2-\ip{d_{t+1}^y}{\widetilde d_t^y}
		=-\frac12\norm{d_{t+1}^y}^2
		-\ip{d_{t+1}^y}{\widetilde d_t^y-d_{t+1}^y},
	\end{equation*}
	we obtain
	\begin{align}
		(1+h\gamma)\mathcal E_{t+1}-\mathcal E_t={}&-\beta\Bigl[
		f(x_{t+1},y_{t+1})-f(x_t,y_t)
		-h\ip{\nabla_y f(x_{t+1},y_{t+1})}{\widetilde d_t^y}\Bigr]
		\notag\\
		&-h\gamma\beta\Bigl[
		f(x_{t+1},y_{t+1})-f^\star
		-\ip{\nabla_y f(x_{t+1},y_{t+1})}{y_{t+1}-y^\star}\Bigr]
		\notag\\
		&+\Bigl[
		\ip{\widetilde d_t^x}{d_{t+1}^x-d_t^x}
		+\ip{\widetilde d_t^y}
		{\nabla_y f(x_{t+1},y_{t+1})-\nabla_y f(x_t,y_t)}
		\Bigr]
		\notag\\
		&+\frac{h\gamma}{2}\norm{d_{t+1}^x}^2
		-\frac{h\gamma}{2}\norm{d_{t+1}^y}^2
		-h\gamma\ip{d_{t+1}^y}{\widetilde d_t^y-d_{t+1}^y}
		\notag\\
		&+h\gamma^2\ip{2\widetilde d_t^y-d_{t+1}^y}{y_{t+1}-y^\star}
		+h\gamma^3\norm{y_{t+1}-y^\star}^2
		\notag\\
		&-\ip{\widetilde d_t^x-d_{t+1}^x}{d_{t+1}^x-d_t^x}
		-\ip{\widetilde d_t^y-d_{t+1}^y}{d_{t+1}^y-d_t^y}
		\notag\\
		&-\frac12\left(
		\norm{d_{t+1}^x-d_t^x}^2+\norm{d_{t+1}^y-d_t^y}^2\right)
		\notag\\
		&-h\gamma\ip{v_{t+1}-v_t}{\widetilde d_t^y}
		-h^2\gamma^2\norm{\widetilde d_t^y}^2.
		\label{eq:combined-expansion}
	\end{align}
	Strong convexity of $f(\cdot,y_t)$ and strong concavity of
	$f(x_{t+1},\cdot)$, followed by completing the square, yield
	\begin{align*}
		&-\beta\Bigl[f(x_{t+1},y_{t+1})-f(x_t,y_t)
		-h\ip{\nabla_y f(x_{t+1},y_{t+1})}{\widetilde d_t^y}\Bigr]
		-\frac12\norm{d_{t+1}^x-d_t^x}^2\\
		\le{}&h\beta\ip{d_t^x}{\widetilde d_t^x}
		-\frac12\norm{d_{t+1}^x-d_t^x}^2
		-\frac{h^2\beta}{2}\left(
		\mu_x\norm{\widetilde d_t^x}^2
		+\mu_y\norm{\widetilde d_t^y}^2\right)\\
		={}&h\beta\norm{d_{t+1}^x}^2
		+h\beta\ip{d_{t+1}^x}{\widetilde d_t^x-d_{t+1}^x}
		-h\beta\ip{d_{t+1}^x-d_t^x}{\widetilde d_t^x}\\
		&-\frac12\norm{d_{t+1}^x-d_t^x}^2
		-\frac{h^2\beta}{2}\left(
		\mu_x\norm{\widetilde d_t^x}^2
		+\mu_y\norm{\widetilde d_t^y}^2\right)\\
		={}&h\beta\norm{d_{t+1}^x}^2
		+h\beta\ip{d_{t+1}^x}{\widetilde d_t^x-d_{t+1}^x}
		-\frac12\norm{d_{t+1}^x-d_t^x+h\beta\widetilde d_t^x}^2\\
		&-\frac{h^2\beta}{2}\Bigl[
		(\mu_x-\beta)\norm{\widetilde d_t^x}^2
		+\mu_y\norm{\widetilde d_t^y}^2\Bigr].
	\end{align*}
	Substituting this bound into \eqref{eq:combined-expansion}, applying
	\eqref{eq:geom3} at $t+1$ and Lemma~\ref{lem:correction-geometry},
	and dropping nonpositive terms using
	\eqref{eq:analysis-parameters} and \eqref{eq:descent-parameters},
	we obtain
	\begin{align}
		(1+h\gamma)\mathcal E_{t+1}-\mathcal E_t\le &h\left(\beta+\frac{\gamma}{2}\right)\norm{d_{t+1}^x}^2
		+h\beta\ip{d_{t+1}^x}{\widetilde d_t^x-d_{t+1}^x}
		-h\mu_x\norm{\widetilde d_t^x}^2
		\notag\\
		&-\frac{h\gamma}{2}\norm{d_{t+1}^y}^2
		-h\gamma\ip{d_{t+1}^y}{\widetilde d_t^y-d_{t+1}^y}
		\notag\\
		&+h\gamma^2\ip{2\widetilde d_t^y-d_{t+1}^y}{y_{t+1}-y^\star}
		+\frac{h\gamma}{2}(2\gamma^2-\beta\mu_y)\norm{y_{t+1}-y^\star}^2
		\notag\\
		&-\frac12\norm{d_{t+1}^x-d_t^x+h\beta\widetilde d_t^x}^2
		-\frac12\norm{d_{t+1}^y-d_t^y}^2
		\notag\\
		&-\ip{\widetilde d_t^x-d_{t+1}^x}{d_{t+1}^x-d_t^x}
		-\ip{\widetilde d_t^y-d_{t+1}^y}{d_{t+1}^y-d_t^y}
		\notag\\
		&-h\gamma\ip{v_{t+1}-v_t}{\widetilde d_t^y}.
		\label{eq:combined-collected}
	\end{align}
	Young's inequality and the parameter bounds give
	\begin{align*}
		&h\left(\beta+\frac{\gamma}{2}\right)\norm{d_{t+1}^x}^2
		+h\beta\ip{d_{t+1}^x}{\widetilde d_t^x-d_{t+1}^x}
		-h\mu_x\norm{\widetilde d_t^x}^2\\
		={}&-h\left(\mu_x-\beta-\frac{\gamma}{2}\right)
		\norm{d_{t+1}^x}^2
		-h(2\mu_x-\beta)
		\ip{d_{t+1}^x}{\widetilde d_t^x-d_{t+1}^x}
		-h\mu_x\norm{\widetilde d_t^x-d_{t+1}^x}^2\\
		\le{}&-h\left(\frac{\mu_x}{2}-\beta-\frac{\gamma}{2}\right)
		\norm{d_{t+1}^x}^2
		+h\mu_x\norm{\widetilde d_t^x-d_{t+1}^x}^2\\
		\le{}&-\frac{h\mu_x}{6}\norm{d_{t+1}^x}^2
		+h\mu_x\norm{\widetilde d_t^x-d_{t+1}^x}^2.
	\end{align*}
	Similarly, completing the square and using the parameter bounds yields
	\begin{align*}
		&-\frac{h\gamma}{2}\norm{d_{t+1}^y}^2
		-h\gamma\ip{d_{t+1}^y}{\widetilde d_t^y-d_{t+1}^y}
		+h\gamma^2\ip{2\widetilde d_t^y-d_{t+1}^y}{y_{t+1}-y^\star}\\
		&\quad+\frac{h\gamma}{2}(2\gamma^2-\beta\mu_y)
		\norm{y_{t+1}-y^\star}^2\\
		={}&-\frac{h\gamma}{4}\norm{d_{t+1}^y}^2
		+h\gamma\norm{\widetilde d_t^y-d_{t+1}^y}^2
		-\frac{h\gamma}{4}
		\norm{2\widetilde d_t^y-d_{t+1}^y
			-2\gamma(y_{t+1}-y^\star)}^2\\
		&\quad+\frac{h\gamma}{2}(4\gamma^2-\beta\mu_y)
		\norm{y_{t+1}-y^\star}^2\\
		\le{}&-\frac{h\gamma}{4}\norm{d_{t+1}^y}^2
		+h\mu_x\norm{\widetilde d_t^y-d_{t+1}^y}^2.
	\end{align*}
	Substituting these estimates into \eqref{eq:combined-collected} gives
	\begin{align}
		(1+h\gamma)\mathcal E_{t+1}-\mathcal E_t
		\le{}&-\frac{h\mu_x}{6}\norm{d_{t+1}^x}^2
		-\frac{h\gamma}{4}\norm{d_{t+1}^y}^2
		+h\mu_x\left(\norm{\widetilde d_t^x-d_{t+1}^x}^2
		+\norm{\widetilde d_t^y-d_{t+1}^y}^2\right)
		\notag\\
		&-\frac12\norm{d_{t+1}^x-d_t^x+h\beta\widetilde d_t^x}^2
		-\frac12\norm{d_{t+1}^y-d_t^y}^2
		\notag\\
		&-\ip{\widetilde d_t^x-d_{t+1}^x}{d_{t+1}^x-d_t^x}
		-\ip{\widetilde d_t^y-d_{t+1}^y}{d_{t+1}^y-d_t^y}
		\notag\\
		&-h\gamma\ip{v_{t+1}-v_t}{\widetilde d_t^y}.
		\label{eq:after-direction-bounds}
	\end{align}

	We next bound the shifted square and the momentum term.
	Young's inequality gives
	\begin{equation*}
		-\frac12\norm{d_{t+1}^x-d_t^x+h\beta\widetilde d_t^x}^2
		\le-\frac38\norm{d_{t+1}^x-d_t^x}^2
		+\frac{3h^2\beta^2}{2}\norm{\widetilde d_t^x}^2.
	\end{equation*}
	For the momentum term, the definition of $d_t^y$, smoothness, and
	the correction updates yield
	\begin{align*}
		\ip{v_{t+1}-v_t}{\widetilde d_t^y}
		={}&\ip{d_{t+1}^y-d_t^y}{\widetilde d_t^y}
		-\ip{\nabla_y f(x_{t+1},y_{t+1})-\nabla_y f(x_t,y_t)}{\widetilde d_t^y}\\
		\ge{}&-\norm{d_{t+1}^y-d_t^y}\norm{\widetilde d_t^y}
		-hL
		\sqrt{\norm{\widetilde d_t^x}^2+\norm{\widetilde d_t^y}^2}
		\norm{\widetilde d_t^y}\\
		\ge{}&-\frac1{8h\gamma}\norm{d_{t+1}^y-d_t^y}^2
		-\frac{hL}{2}\norm{\widetilde d_t^x}^2
		-(hL+2h\gamma)\norm{\widetilde d_t^y}^2.
	\end{align*}
	The last step uses Young's inequality and
	$\sqrt{a^2+b^2}\,b\le a^2/2+b^2$ for $a,b\ge0$.
	Multiplying this bound by $-h\gamma$ and substituting it, together
	with the square bound, into \eqref{eq:after-direction-bounds} gives
	\begin{align}
		(1+h\gamma)\mathcal E_{t+1}-\mathcal E_t
		\le{}&-\frac{h\mu_x}{6}\norm{d_{t+1}^x}^2
		-\frac{h\gamma}{4}\norm{d_{t+1}^y}^2
		+h\mu_x\left(\norm{\widetilde d_t^x-d_{t+1}^x}^2
		+\norm{\widetilde d_t^y-d_{t+1}^y}^2\right)
		\notag\\
		&+\left(\frac{3h^2\beta^2}{2}+\frac{h^2L\gamma}{2}\right)
		\norm{\widetilde d_t^x}^2
		+(h^2L\gamma+2h^2\gamma^2)\norm{\widetilde d_t^y}^2
		\notag\\
		&-\frac38\left(\norm{d_{t+1}^x-d_t^x}^2
		+\norm{d_{t+1}^y-d_t^y}^2\right)
		\notag\\
		&-\ip{\widetilde d_t^x-d_{t+1}^x}{d_{t+1}^x-d_t^x}
		-\ip{\widetilde d_t^y-d_{t+1}^y}{d_{t+1}^y-d_t^y}
		\notag\\
		\le{}&\left(-\frac{h\mu_x}{6}+3h^2\beta^2+h^2L\gamma\right)
		\norm{d_{t+1}^x}^2
		\notag\\
		&+\left[-\frac{h\gamma}{4}+\frac54(h^2L\gamma+2h^2\gamma^2)\right]
		\norm{d_{t+1}^y}^2
		\notag\\
		&+(h\mu_x+3h^2\beta^2+h^2L\gamma)
		\norm{\widetilde d_t^x-d_{t+1}^x}^2
		\notag\\
		&+\left[h\mu_x+5(h^2L\gamma+2h^2\gamma^2)\right]
		\norm{\widetilde d_t^y-d_{t+1}^y}^2
		\notag\\
		&-\frac38\left(\norm{d_{t+1}^x-d_t^x}^2
		+\norm{d_{t+1}^y-d_t^y}^2\right)
		\notag\\
		&-\ip{\widetilde d_t^x-d_{t+1}^x}{d_{t+1}^x-d_t^x}
		-\ip{\widetilde d_t^y-d_{t+1}^y}{d_{t+1}^y-d_t^y},
		\label{eq:after-momentum-bound}
	\end{align}
	where the last inequality uses
	$\norm{r+s}^2\le2\norm r^2+2\norm s^2$ and
	$\norm{r+s}^2\le5\norm r^2/4+5\norm s^2$
	for the $x$- and $y$-direction terms, respectively.
	By \eqref{eq:analysis-parameters} and \eqref{eq:descent-parameters},
	the coefficients of $\norm{d_{t+1}^x}^2$ and
	$\norm{d_{t+1}^y}^2$ in \eqref{eq:after-momentum-bound} are
	nonpositive, and both mismatch coefficients are at most
	$\frac54h\mu_x$. Hence,
	\begin{align}
		(1+h\gamma)\mathcal E_{t+1}-\mathcal E_t
		\le{}&\frac54h\mu_x\left(
		\norm{\widetilde d_t^x-d_{t+1}^x}^2
		+\norm{\widetilde d_t^y-d_{t+1}^y}^2\right)
		\notag\\
		&-\frac38\left(
		\norm{d_{t+1}^x-d_t^x}^2
		+\norm{d_{t+1}^y-d_t^y}^2\right)
		\notag\\
		&-\ip{\widetilde d_t^x-d_{t+1}^x}{d_{t+1}^x-d_t^x}
		-\ip{\widetilde d_t^y-d_{t+1}^y}{d_{t+1}^y-d_t^y}.
		\label{eq:pre-absorb}
	\end{align}
	Finally, applying Lemma~\ref{lem:defect} and Cauchy--Schwarz to
	\eqref{eq:pre-absorb}, and using $h\mu_x\le2/15$, yields
	\eqref{eq:main-descent}, since $1+h\gamma=(1+\alpha)/(2\alpha)$.
\end{proof}
The coercivity and one-step descent estimates imply last-iterate linear
convergence in squared Euclidean distance.

\begin{theorem}\label{thm:linear}
	Suppose that Assumption \ref{ass:main} holds and that the parameters satisfy
	\eqref{eq:analysis-parameters}. Let $(x_t,y_t,v_t)$ be generated by
	Algorithm \ref{alg:main}. Then, for every integer $T\ge0$,
	\begin{equation}\label{eq:linear-rate}
		D_T\le \frac{32L^2}{\mu_y\sqrt{\mu_x\mu_y}}
		\left(1+\frac{h\sqrt{\mu_x\mu_y}}{16}\right)^{-T}D_0.
	\end{equation}
	Consequently, for $0<\varepsilon<1$, any integer $T$ satisfying
	\begin{equation}\label{eq:iteration-complexity}
		T\ge\frac{32}{h\sqrt{\mu_x\mu_y}}
		\log\frac{32L^2}{\mu_y\sqrt{\mu_x\mu_y}\,\varepsilon}
	\end{equation}
	ensures that $(x_T,y_T)$ is an $\varepsilon$-accurate relative solution
	in the sense of Definition~\ref{def:accuracy}, namely
	$D_T\le\varepsilon D_0$.
\end{theorem}
\begin{proof}
	By Lemma~\ref{lem:coercivity} and $\mu_x\ge\mu_y$,
	\begin{equation*}
		\mathcal E_t\ge\frac{\sqrt{\mu_x\mu_y}}{32}
		\bigl(\mu_x\norm{x_t-x^\star}^2+\mu_y\norm{y_t-y^\star}^2\bigr)
		\ge\frac{\mu_y\sqrt{\mu_x\mu_y}}{32}D_t\ge0.
	\end{equation*}
	Dropping the nonpositive right-hand side of \eqref{eq:main-descent}
	gives
	\begin{equation}\label{eq:linear-contraction}
		\frac{1+\alpha}{2\alpha}\mathcal E_{t+1}\le\mathcal E_t.
	\end{equation}
	Iterating this inequality and using the upper bound on $\alpha$ in
	\eqref{eq:analysis-parameters} yields
	\begin{equation*}
		\mathcal E_T\le
		\left(\frac{1+\alpha}{2\alpha}\right)^{-T}\mathcal E_0
		\le\left(1+\frac{h\sqrt{\mu_x\mu_y}}{16}\right)^{-T}\mathcal E_0.
	\end{equation*}
	Since $\nabla f(x^\star,y^\star)=0$, smoothness implies
	\begin{equation*}
		\norm{\nabla f(x_0,y_0)}^2\le L^2D_0,
		\qquad
		\bigl|f(x_0,y_0)-f^\star\bigr|\le\frac L2D_0.
	\end{equation*}
	Substituting these bounds and $v_0=0$ into \eqref{eq:lyapunov}, and
	using \eqref{eq:analysis-parameters}, gives
	\begin{equation*}
		\mathcal E_0\le
		\left[\frac{L^2}{2}+\frac{(1-\alpha+2\rho)L}{4\alpha h}
		+\frac{(1-\alpha)^2}{4\alpha^2h^2}\right]D_0
		\le L^2D_0.
	\end{equation*}
	Combining these estimates proves \eqref{eq:linear-rate}.
	Finally, since $0<h\sqrt{\mu_x\mu_y}/16<1$, applying
	$\log(1+s)\ge s/2$ for $0\le s\le1$ to \eqref{eq:linear-rate}
	shows that \eqref{eq:iteration-complexity} suffices to ensure
	$D_T\le\varepsilon D_0$.
\end{proof}
\begin{remark}
	Choosing $h=\Theta(1/L)$ subject to \eqref{eq:analysis-parameters},
	the bound \eqref{eq:iteration-complexity} yields an iteration complexity of
	\begin{equation*}
		\mathcal{O}\!\left(\frac{L}{\sqrt{\mu_x\mu_y}}
		\log\frac{2\kappa_x\kappa_y}{\varepsilon}\right)
		=\mathcal{O}\!\left(\sqrt{\kappa_x\kappa_y}
		\log\frac{2\kappa_x\kappa_y}{\varepsilon}\right),
	\end{equation*}
	where $\kappa_x:=L/\mu_x$ and $\kappa_y:=L/\mu_y$.
	The logarithmic factor follows from
	$32L^2/(\mu_y\sqrt{\mu_x\mu_y})
	=32\kappa_x\kappa_y\sqrt{\kappa_y/\kappa_x}
	\le32(\kappa_x\kappa_y)^{3/2}$
	and $\kappa_x,\kappa_y\ge1$.
	Each iteration of Algorithm \ref{alg:main} requires a constant number of new
	full-gradient evaluations. Thus, the full-gradient oracle complexity
	for obtaining an $\varepsilon$-accurate relative solution has the same
	order.
\end{remark}

\section{Numerical experiments}\label{sec:experiments}

We evaluate PCE-DM on regularized linear regression and AUC
maximization. The baselines comprise three single-loop methods:
gradient descent--ascent (GDA), extragradient (EG), and optimistic
gradient descent--ascent (OGDA), together with two nested methods:
Catalyst-EG~\cite[Chapter~2]{yang2023thesis} and
FOAM~\cite{kovalev2022}.
We measure convergence by the relative squared distance $D_t/D_0$
and compare computational cost in terms of full-gradient queries
and computation time.

For each problem instance, all methods start from the same initial
point, and each method uses the best-performing configuration found
in its respective parameter search. Full-gradient query counts include
initialization and all inner-loop evaluations. Computation time includes
initialization, gradient evaluations, variable updates, and required
inner stopping checks, but excludes outer diagnostic evaluations,
logging, and data preparation. Each reported time is measured in a
single run for the corresponding method and instance.

\subsection{Regularized linear regression}\label{subsec:exp-ridge}
We consider the unconstrained saddle formulation of ridge regression,
\begin{equation}\label{eq:exp-ridge}
	\min_{x\in\R^{50}}\max_{y\in\R^{50}}
	\frac{\mu_x}{2}\norm{x}^2+y^\top Bx-c^\top y
	-\frac{\mu_y}{2}\norm{y}^2.
\end{equation}
We randomly generate $B\in\R^{50\times50}$ and set
$\mu_x=1$ and $\mu_y=1/r$, where
$r\in\{1,10,100,1000,10000\}$ controls the curvature imbalance.

Figure~\ref{fig:ridge} shows convergence against full-gradient queries,
and Table~\ref{tab:ridge-queries} reports the cost of reaching
$D_t/D_0\le10^{-6}$.
GDA requires the fewest queries and the least time at $r=1$.
EG and OGDA use more queries than GDA at every tested ratio.
Among the nested methods, Catalyst-EG improves on GDA's query count
only at $r=10000$, while FOAM becomes the most query-efficient
baseline at $r=1000$ and $r=10000$.

PCE-DM achieves the lowest query count and measured computation time
for every tested $r\ge10$. Its query advantage over the best baseline
increases across these ratios. At $r=10000$, PCE-DM requires
$1381$ queries and $36.025$~ms, compared with $11987$ queries and
$468.783$~ms for FOAM, the best baseline on both metrics.
This corresponds to approximately $88.5\%$ fewer queries and a
$13.0$-fold speedup in the recorded run. These results suggest that
PCE-DM is less sensitive to curvature imbalance on this test family.
\begin{figure}
	\centering
	\includegraphics[width=0.94\linewidth]{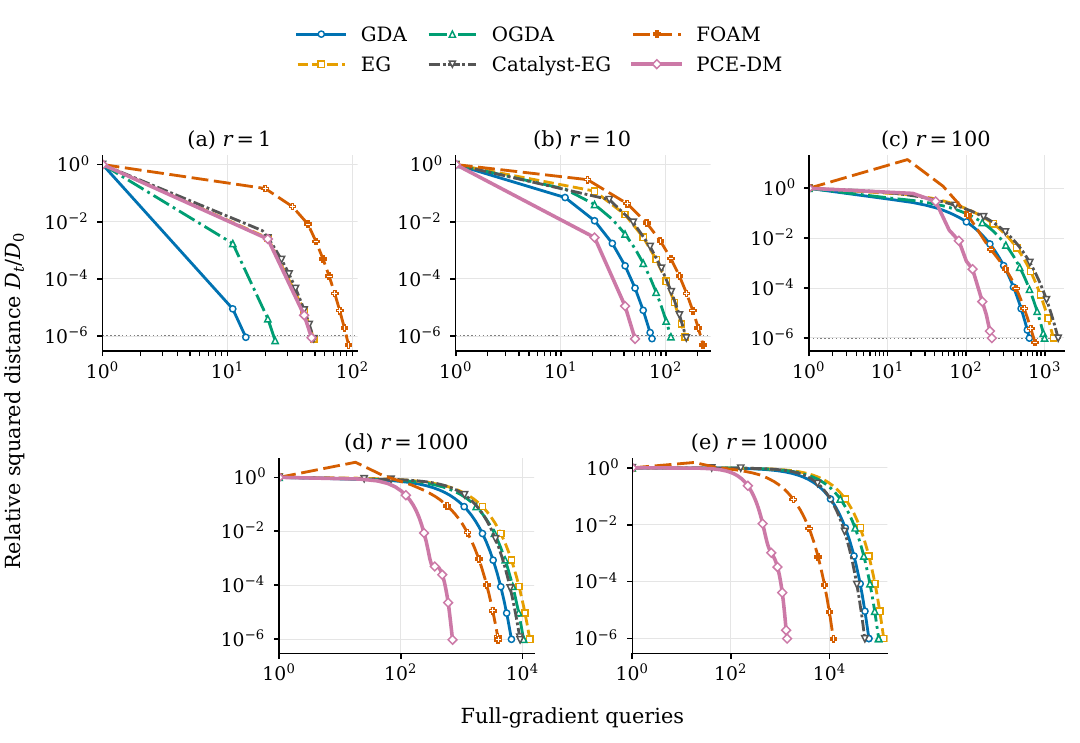}
	\caption{Regularized linear regression:
		relative squared distance $D_t/D_0$ versus full-gradient queries.
		Panels (a)--(e) correspond to $r=1,10,100,1000,10000$.
		Both axes use logarithmic scales; dotted lines indicate
		$D_t/D_0=10^{-6}$.}
	\label{fig:ridge}
\end{figure}
\begin{table}[!htbp]
	\centering
	\caption{Full-gradient queries and computation time (ms) required to
		reach $D_t/D_0\le10^{-6}$ in regularized linear regression.
		Both metrics include initialization and all inner-loop work.
		Bold entries indicate the smallest value in each row.}
	\label{tab:ridge-queries}
	\small
	\begin{tabular*}{\linewidth}{@{\extracolsep{\fill}}rlrrrrrr@{}}
		\toprule
		$r$ & Metric & GDA & EG & OGDA & Catalyst-EG & FOAM & PCE-DM\\
		\midrule
		1     & Queries & \textbf{14} & 49 & 24 & 49 & 93 & 47\\
		& Time & \textbf{0.305} & 1.345 & 0.629 & 1.780 & 6.146 & 2.192\\
		\addlinespace[2pt]
		10    & Queries & 74 & 153 & 112 & 157 & 227 & \textbf{51}\\
		& Time & 1.418 & 2.830 & 2.619 & 4.009 & 9.313 & \textbf{1.207}\\
		\addlinespace[2pt]
		100   & Queries & 639 & 1289 & 996 & 1483 & 752 & \textbf{213}\\
		& Time & 13.496 & 26.155 & 23.985 & 56.492 & 34.116 & \textbf{4.796}\\
		\addlinespace[2pt]
		1000  & Queries & 6593 & 13199 & 10417 & 9175 & 4002 & \textbf{709}\\
		& Time & 132.159 & 255.474 & 244.055 & 251.218 & 178.139 & \textbf{16.804}\\
		\addlinespace[2pt]
		10000 & Queries & 62835 & 125683 & 99210 & 52057 & 11987 & \textbf{1381}\\
		& Time & 1241.479 & 2522.581 & 2401.702 & 1354.875 & 468.783 & \textbf{36.025}\\
		\bottomrule
	\end{tabular*}
\end{table}
\subsection{AUC maximization}\label{subsec:exp-auc}

We consider the regularized AUC saddle model of
Liu and Luo~\cite[Section~5.1]{liuluo2022}.
Given labeled samples $(a_i,b_i)\in\R^m\times\{-1,1\}$, define
$p=N^{-1}\sum_{i=1}^N\mathbf{1}_{\{b_i=1\}}$ and
$z=(w,u,v)\in\R^{m+2}$. The model is
\begin{equation}\label{eq:exp-auc}
	\begin{aligned}
		\min_{z\in\R^{m+2}}\max_{s\in\R}\quad
		F(z,s)={}&\frac{\lambda}{2}\norm{z}^2-p(1-p)s^2\\
		&+\frac{p}{N}\sum_{i:b_i=-1}\bigl[(a_i^\top w-v)^2+2(1+s)a_i^\top w\bigr]\\
		&+\frac{1-p}{N}\sum_{i:b_i=1}\bigl[(a_i^\top w-u)^2-2(1+s)a_i^\top w\bigr],
	\end{aligned}
\end{equation}
with $\lambda=100/N$.
We use the complete official LIBSVM training files for a9a
($N=32561$, $m=123$) and cod-rna ($N=59535$, $m=8$).
No additional feature preprocessing is applied to a9a; each cod-rna
feature is divided by its maximum absolute value over the training
samples. For this experiment, we measure convergence by the full-gradient
norm $\norm{\nabla F(z,s)}$ and report the cost of reaching
$\norm{\nabla F(z,s)}\le10^{-6}$.

Figure~\ref{fig:auc} and Table~\ref{tab:auc-results} show that all six
methods reach the prescribed tolerance on both datasets.
PCE-DM requires the fewest full-gradient queries and the least measured
computation time, followed by GDA. Catalyst-EG outperforms EG and OGDA
on both metrics, but neither nested method improves on GDA or PCE-DM.
FOAM uses slightly fewer queries than Catalyst-EG on a9a but takes
more time; on cod-rna, it requires both more queries and more time.

Compared with GDA, the best baseline on both datasets, PCE-DM reduces
the query count from $9593$ to $1585$ on a9a and from $8588$ to $2789$
on cod-rna, corresponding to reductions of approximately $83.5\%$
and $67.5\%$. The measured times decrease from $239.592$ to
$40.394$~ms and from $186.133$ to $67.873$~ms, respectively.
These correspond to speedups of approximately $5.9$ and $2.7$ in
the recorded runs, showing that PCE-DM's query savings also translate
into lower computation time on these two instances.
\begin{figure}
	\centering
	\includegraphics[width=\linewidth]{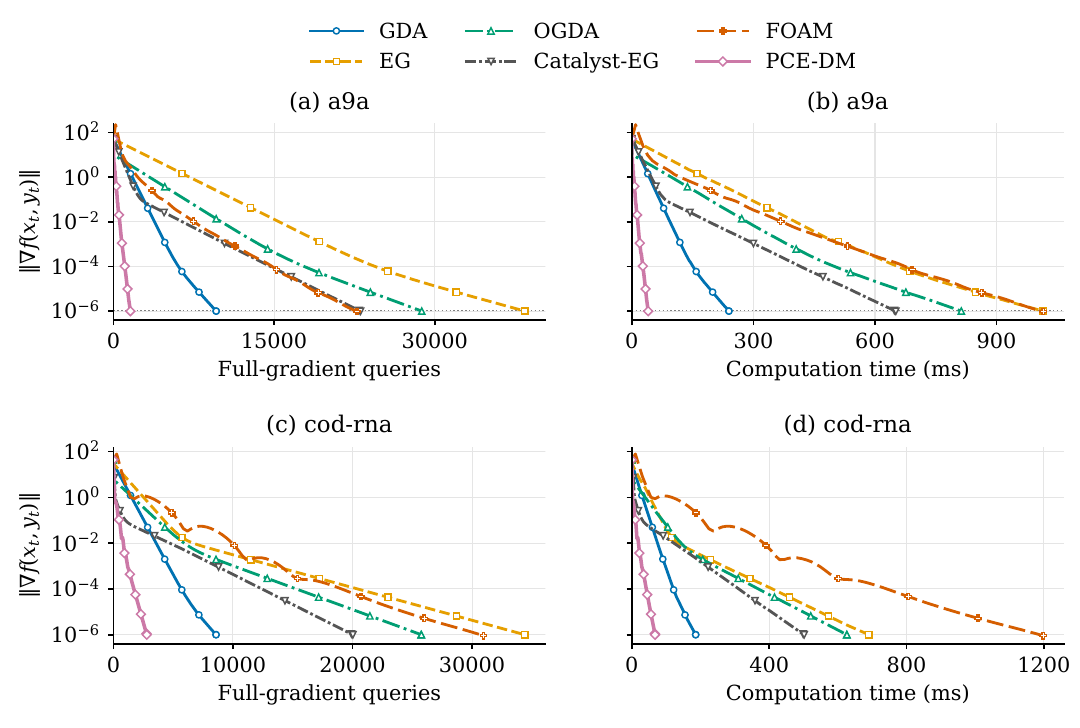}
	\caption{AUC maximization on a9a (top) and cod-rna (bottom):
		full-gradient norm versus full-gradient queries (left) and computation
		time (right). All methods use the same objective and initial point
		within each dataset. The vertical axes use logarithmic scales;
		dotted lines indicate the tolerance $10^{-6}$.}
	\label{fig:auc}
\end{figure}
\begin{table}[!htbp]
	\centering
	\caption{AUC maximization: full-gradient queries and computation
		time (ms) required to reach $\norm{\nabla F(z,s)}\le10^{-6}$.
		Both metrics include initialization and all inner-loop work.
		Bold entries indicate the smallest value in each column.}
	\label{tab:auc-results}
	\small
	\begin{tabular*}{\linewidth}{@{\extracolsep{\fill}}lrrrr@{}}
		\toprule
		& \multicolumn{2}{c}{a9a} & \multicolumn{2}{c}{cod-rna}\\
		\cmidrule(lr){2-3}\cmidrule(l){4-5}
		Method & Queries & Time (ms) & Queries & Time (ms)\\
		\midrule
		GDA         & 9593          & 239.592         & 8588          & 186.133\\
		EG          & 38401         & 1015.060        & 34411         & 690.528\\
		OGDA        & 28782         & 812.979         & 25760         & 626.530\\
		Catalyst-EG & 23065         & 651.284         & 20029         & 501.871\\
		FOAM        & 22744         & 1015.045        & 30964         & 1198.368\\
		PCE-DM      & \textbf{1585} & \textbf{40.394} & \textbf{2789} & \textbf{67.873}\\
		\bottomrule
	\end{tabular*}
\end{table}

\section{Conclusion}\label{sec:conclusion}

PCE-DM achieves last-iterate linear convergence for smooth strongly
convex--strongly concave minimax problems with general nonlinear
coupling. Its fixed explicit updates require two new full-gradient
evaluations per iteration and attain the optimal condition-number
dependence up to logarithmic factors, with oracle complexity
$\mathcal{O}(\sqrt{\kappa_x\kappa_y}\log(2\kappa_x\kappa_y/\eps))$.
The analysis uses the shared feedback history to control the
predictor--corrector mismatch, which is absorbed by negative
increment terms in a coercive Lyapunov estimate.

PCE-DM uses the fewest gradient queries and the least measured time
on the tested regression instances with unequal curvatures and on
both AUC datasets, while GDA performs best in the balanced regression
case. These results indicate particular benefits under curvature
imbalance. Future work includes adaptive parameter selection, extensions to
constrained and stochastic problems, and experiments with nonlinear
coupling and larger datasets.

\bmhead{Acknowledgments}
This work was supported by the National Key R\&D Program of China (Grant Nos.~2025YFA1017801 and 2025YFA1017800) and the National Natural
Science Foundation of China (Grant No.~12471294).
%
%

\section*{Declarations}
%
%
\begin{itemize}
	\item {\bf Conflict of interest} The authors have no relevant financial or non-financial interests to disclose.
	\item {\bf Data availability} Data sharing not applicable to this article as no datasets were generated or analyzed during the current study.
\end{itemize}

\end{document}